\documentclass[12pt]{amsart}
\usepackage{amscd,amssymb,amsthm,amsmath,amssymb,mathrsfs,enumerate,longtable,pdflscape}
\usepackage[matrix,arrow,curve]{xy}
\usepackage[margin=2cm]{geometry}
\usepackage[normalem]{ulem}

\newtheorem*{theorem*}{Theorem}

\newtheorem*{proposition*}{Proposition}

\newtheorem*{lemma*}{Lemma}

\newtheorem*{corollary*}{Corollary}

\newtheorem*{problem*}{Problem}
\newtheorem*{claim*}{Claim}
\newtheorem*{mainresult*}{Main Result}

\theoremstyle{definition}

\newtheorem*{construction*}{Construction}
\newtheorem*{example*}{Example}

\newtheorem*{observation*}{Observation}

\theoremstyle{remark}

\newtheorem*{remark*}{Remark}

\makeatletter\@addtoreset{equation}{section} \makeatother

\def\P{\mathbb{P}}

\author{Ivan Cheltsov, Igor Krylov, Jesus Martinez-Garcia, Evgeny Shinder}

\thanks{Throughout this paper, all varieties are assumed to be projective and defined over~$\mathbb{C}$.}

\title{On maximally non-factorial nodal Fano threefolds}

\begin{document}

\begin{abstract}
We classify non-factorial nodal Fano threefolds with $1$ node and class group of rank $2$.
\end{abstract}


\address{ \emph{Ivan Cheltsov}\newline
\textnormal{University of Edinburgh, Edinburgh, Scotland}\newline i.cheltsov@ed.ac.uk}

\address{\emph{Igor Krylov}
\newline
\textnormal{Institute for Basic Science, Pohang, Korea}\newline	Igor@Krylov.su}

\address{ \emph{Jesus Martinez-Garcia}
\newline
\textnormal{University of Essex, Colchester, England}\newline jesus.martinez-garcia@essex.ac.uk}

\address{ \emph{Evgeny Shinder}
\newline
\textnormal{University of Sheffield, Sheffield, England}
\newline
\textnormal{University of Bonn, Bonn, Germany}
\newline eugene.shinder@gmail.com}

\maketitle

Let $X$ be a
Fano threefold that has at worst isolated ordinary double points (nodes).
Then~both the Picard group $\mathrm{Pic}(X)$ and the class group $\mathrm{Cl}(X)$ are torsion-free of finite rank,
and $\mathrm{rk}\,\mathrm{Cl}(X)-\mathrm{rk}\,\mathrm{Pic}(X)$ is known as the~\emph{defect} of the~threefold $X$ \cite{Clemens,Cynk,CynkRams,Kaloghiros}.
If the~defect is zero, we say that $X$ is \emph{factorial} \cite{Cheltsov2009,Cheltsov2010}.
Factoriality imposes significant  constraints on the~geometry of the~Fano threefold \cite{CheltsovPark2010,CheltsovPrzyjalkowskiShramov2019,Mella,Shramov}.

It is well known that the~defect of  $X$ does not exceed the~number of its singular points (see e.g. \cite[Corollary 3.8]{KalckPavicShinder}).
If
$$
\mathrm{rk}\,\mathrm{Cl}(X)-\mathrm{rk}\,\mathrm{Pic}(X)=|\mathrm{Sing}(X)|,
$$
then we say that $X$ is 
\emph{maximally non-factorial}.
This property is also called $\mathbb{Q}$-maximal non-factoriality; 
see 
\cite[Proposition 6.13]{KuznetsovShinder}
and
\cite[Proposition A.14]{KuznetsovShinder2}
for various ways to define it for a nodal Fano threefold $X$.
By definition, if $X$ has a single node, then $X$ is maximally non-factorial if and only if it is non-factorial.
Let us give the simplest~example of a non-factorial threefold with one node.

\begin{example*}
Let $X$ be the quadric cone in $\mathbb{P}^4$ with one node.
Then $X$ is a maximally non-factorial nodal Fano threefold.
Let $\eta\colon\widetilde{X}\to X$ be the~blow up at the~singular point of the~threefold~$X$,
and let $E$ be the~$\eta$-exceptional surface.
Then $E\cong\mathbb{P}^1\times\mathbb{P}^1$ and $E\vert_{E}\cong\mathcal{O}_{E}(-1,-1)$,
and there exists the~following commutative diagram:
$$
\xymatrix{
&&\widetilde{X}\ar[rd]^{\varphi_2}\ar[ld]_{\varphi_1}\ar[dd]_{\eta}&&\\
&X_1\ar[ld]_{\pi_1}\ar[rd]^{\phi_1}&&X_2\ar[ld]_{\phi_2}\ar[rd]^{\pi_2}&\\
\mathbb{P}^1&&X&&\mathbb{P}^1}
$$
where $\varphi_1$ and $\varphi_2$ are contractions of the~surface $E$ to curves such that $\varphi_2\circ\varphi_1^{-1}$ is an Atiyah flop,
both $\phi_1$ and $\phi_2$ are small projective resolutions, and both $\pi_1$ and $\pi_2$ are $\mathbb{P}^2$-bundles.
\end{example*}

While maximally non-factorial nodal Fano threefolds are important in birational geometry  because of their rich geometry, 
they also play a central
role in the recent
study of the derived categories of coherent sheaves for singular varieties
(this in turn allows the study of the birational geometry with a very different toolset, such as stability conditions). 
Indeed, 
maximally non-factorial nodal Fano threefolds are very special from the~perspective of derived categories of coherent sheaves,
in particular their~derived categories can often be separated into a smooth proper part and a singular part \cite{KalckPavicShinder,PavicShinder,Xie,KuznetsovShinder,KuznetsovShinder2}. 

Let us explain the connection to derived categories in some more
detail.
In \cite{KalckPavicShinder, PavicShinder, Xie} the authors, inspired by the work of Kawamata \cite{Kawamata},
introduced and studied maximal non-factoriality 
of del Pezzo threefolds (a subset consisting of 8 of the 105 families of Fano threefolds). They proved that a del Pezzo threefold is maximally non-factorial if and only if its derived category admits a Kawamata semiorthogonal 
decomposition, that is an admissible semiorthogonal decomposition into a perfect part and derived categories of singular finite-dimensional algebras. 
It is thus natural to ask whether being a maximally non-factorial Fano threefold is a sufficient condition for  the existence of a Kawamata decomposition.
 The proof of \cite{PavicShinder} relied on the classification of del Pezzo threefolds which are maximally non-factorial. Thus, in order to study Kawamata decompositions it is natural to have a 
 classification of maximally non-factorial Fano threefolds.
 
A slightly weaker notion of categorical absorption of singularities was introduced in \cite{KuznetsovShinder, KuznetsovShinder2}. 
By \cite[Corollary 6.17]{KuznetsovShinder},
every maximally non-factorial Fano threefold with one ordinary double point admits a categorical absorption of singularities;
the converse is also true, and holds for any number of nodes \cite[Proposition 6.12]{KuznetsovShinder}.
A highlight of this theory 
in \cite{KuznetsovShinder2} is the deformation between
the main components of the derived categories of
one-nodal prime Fano threefolds of genus $2d+2$
and del Pezzo threefolds of rank one and degree $d$,
for $d \in \{1, 2, 3, 4, 5\}$, which solves
the so-called
Fano threefold conjecture of Kuznetsov.

There is also a consequence of maximal
non-factoriality to intermediate Jacobians.
Namely, in some sense, a maximally non-factorial nodal Fano threefold $X$ 
has
a smooth projective 
intermediate Jacobian, so that  singularities of $X$
can be ignored from the Hodge theory perspective.
The precise statement \cite[Proposition A.16]{KuznetsovShinder2}
is that a family of smooth Fano threefolds degenerating to
a 1-nodal
maximally non-factorial Fano threefold 
has a smooth projective
family of
intermediate Jacobians, i.e. no
actual degeneration takes place in the middle degree cohomology.

On the other hand,
maximally non-factorial Fano threefolds are rather rare among all nodal Fano threefolds. 
Motivated by the recent
advances in derived categories of singular Fano threefolds, we pose the~following problem.

\begin{problem*}
Classify all maximally non-factorial nodal Fano threefolds.
\end{problem*}

The goal of this paper is to partially solve this problem.
Namely, we aim to classify maximally non-factorial nodal Fano threefolds of Picard rank one that have exactly one singular point (node). This case is particularly well behaved from the viewpoint of birational geometry, see the chain of equivalences in \cite[Proposition~6.13]{KuznetsovShinder} that applies only when one singular point is present.

Now, we are ready to present the~main result of this paper.
To do this, we suppose  that
\begin{itemize}
\item the~nodal Fano threefold $X$ has one node,
\item the~rank of the~Picard group $\mathrm{Pic}(X)$ is one,
\item the~rank of the~class group $\mathrm{Cl}(X)$ is two.
\end{itemize}
Let $\eta\colon\widetilde{X}\to X$ be the~blow up of the~node of the~threefold $X$, let $E$ be the~$\eta$-exceptional surface.
Then $\widetilde{X}$ is smooth, $E\cong\mathbb{P}^1\times\mathbb{P}^1$, $E\vert_{E} \simeq \mathcal{O}_{E}(-1,-1)$,
and it follows from \cite{Corti} that $X$ uniquely determines the~following Sarkisov link:
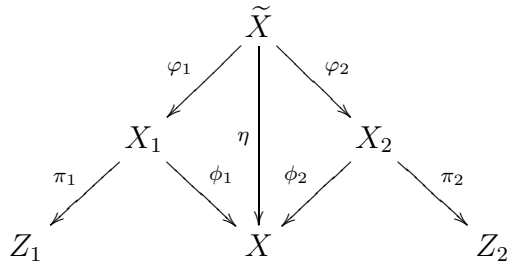
\begin{equation}
\label{equation:link}\tag{$\bigstar$}
\xymatrix{
&&\widetilde{X}\ar[rd]^{\varphi_2}\ar[ld]_{\varphi_1}\ar[dd]_{\eta}&&\\
&X_1\ar[ld]_{\pi_1}\ar[rd]^{\phi_1}&&X_2\ar[ld]_{\phi_2}\ar[rd]^{\pi_2}&\\
Z_1&&X&&Z_2}
\end{equation}
where $\varphi_1$ and $\varphi_2$ are contractions of the~surface $E$ to curves such that $\varphi_2\circ\varphi_1^{-1}$ is an Atiyah flop,
both $\phi_1$ and $\phi_2$ are small projective resolutions, and both $\pi_1$ and $\pi_2$ are extremal contractions \cite{Mori}.
Note that $-K_{X_1}\sim\phi_1^*(-K_{X})$ and  $-K_{X_2}\sim\phi_2^*(-K_{X})$, so that
$$
-K_{X_1}^3=-K_{X_2}^3=-K_{X}^3.
$$

It follows from \cite{Namikawa,JahnkeRadloff2011} that $X$ admits a smoothing $X\rightsquigarrow X_s$, where $X_s$ is a smooth Fano~threefold,
$-K_X^3=-K_{X_s}^3$, and the~rank of the~Picard group $\mathrm{Pic}(X_s)$ is $1$.
We also know from \cite{Clemens} that
\begin{equation}
\label{equation:h-1-2}\tag{$\maltese$}
h^{1,2}(\widetilde X)=h^{1,2}(X_1)=h^{1,2}(X_2)=h^{1,2}(X_s),
\end{equation}
which imposes a significant constraint on the~link \eqref{equation:link}.
We set 
$$
d=-K_X^3, \quad h^{1,2}=h^{1,2}(X_s),
$$ 
and
$$
I=\max\big\{n\in\mathbb{Z}_{>0}\ \text{such that}\ -K_{X_{s}}\sim nH\ \text{for}\ H\in\mathrm{Pic}(X_s)\big\}.
$$
Then $I$ is the~\emph{index} of the~Fano threefold $X_s$, which is also the~index of the~Fano threefold $X$ \cite{JahnkeRadloff2011}.

In the~remaining part of this paper, we prove the~following theorem.

\begin{theorem*}
There are exactly 17 types of  non-factorial
Fano threefolds of Picard rank one with one node.
All possibilities for \eqref{equation:link}, up to swapping the left and right sides of the diagram, are described in the~table at the~end of the~paper.
\end{theorem*}

Each Sarkisov link in the~table exists and can be described explicitly.
For the reader's convenience, we provide the relevant references in the~table.
For the particular case of $-K_X^3=22$, the result
is proved in \cite{Prokhorov16}. 
A. Kuznetsov and Y. Prokhorov have independently obtained the same classification of non-factorial Fano threefolds of Picard rank one  \cite{KuznetsovProkhorov-Fano}.

\begin{remark*}
It should be pointed out that it follows from our classification that one-nodal maximally non-factorial
degenerations of smooth Fano threefolds of Picard rank one (if any) have the same rationality as their smoothing
(in the cases $\mathbf{2}$ and $\mathbf{7}$ in the table we need to assume that the smooth Fano threefolds are general).
Indeed, this can be verified case by case, using the rationality results from
\cite{Beauville,CheltsovShramovPrzyjalkowski2005,ClemensGriffiths,Grinenko,Prokhorov2018,Voisin}.
\end{remark*}

\begin{observation*}
If $X$ is a del Pezzo threefold ($I=2$) of Picard rank one with 
$-K_X^3\leqslant 32$,
then the~nodal Fano threefold $X$ is never maximally non-factorial.
This follows from from the defect computation
\cite{Cynk, CynkRams}, see
\cite[Corollary 2.5]{PavicShinder}.
Therefore, the only options for $X$ when $I>1$ are these two Fano threefolds:
\begin{itemize}
\item the nodal quadric threefold in $\mathbb{P}^4$ ($I=3$, $-K_X^3=54$, the~Sarkisov link $\mathbf{17}$);
\item a quintic del Pezzo threefold ($I=2$, $-K_X^3=40$, the~Sarkisov link $\mathbf{16}$).
\end{itemize}
\end{observation*}

We prove the theorem by analyzing the possible links \eqref{equation:link} in the following order:
\begin{enumerate}
\item $\pi_1$ is a del Pezzo fibration, and $\pi_2$ is arbitrary;
\item both $\pi_1$ and $\pi_2$ are birational;
\item $\pi_1$ is a conic bundle and $\pi_2$ is arbitrary.
\end{enumerate}
These cover all possible  Mori fiber spaces arising in~\eqref{equation:link}, up to swapping $\pi_1$ and $\pi_2$.

Note that all possibilities for the smooth Fano variety $X_s$ are known and can be found in \cite{IskovskikhProkhorov}.
Using this classification, we list the~possible values of $h^{1,2}$ as follows.

\begin{center}
\renewcommand\arraystretch{1.4}
\begin{tabular}{|c||c|c|c|c|c|c|c|c|c|c|}
  \hline
$(d,I)$  & $(2,1)$ & $(4,1)$ & $(6,1)$ & $(8,1)$ & $(10,1)$ & $(12,1)$ & $(14,1)$, & $(16,1)$ & $(18,1)$ & $(22,1)$\\
  \hline
$h^{1,2}$  & $52$ & $30$ & $20$ & $14$ & $10$ & $7$ & $5$ & $3$ & $2$ & $0$\\
  \hline
\end{tabular}
\end{center}

\begin{center}
\renewcommand\arraystretch{1.4}
\begin{tabular}{|c||c|c|c|c|c|c|c|c|c|c|c|c|c|c|c|c|c|}
  \hline
$(d,I)$  & $(8,2)$ & $(16,2)$ & $(24,2)$ & $(32,2)$ & $(40,2)$ & $(54,3)$ & $(64,4)$\\
  \hline
$h^{1,2}$ & $21$ & $10$ & $5$ & $2$ & $0$ & $0$ & $0$\\
  \hline
\end{tabular}
\end{center}

Possibilities for \eqref{equation:link} are studied in \cite{ArapCutroneMarshburn,BlancLamy,CutroneLimarziMarshburn,CutroneMarshburn,Fukuoka2017,Fukuoka2019,JahnkePeternell,JahnkePeternellRadloff2005,JahnkePeternellRadloff2011,JahnkeRadloff2006,JahnkeRadloff2011,Kaloghiros2012,Prokhorov16,Prokhorov2019,Prokhorov2023,Takeuchi1989,Takeuchi2022,Yasutake}.
Using some of these results, we immediately obtain the~following corollary.

\begin{corollary*}
Suppose that $\pi_1$ is a fibration into del Pezzo surfaces.
Then \eqref{equation:link} is one of the~links
\begin{center}
$\mathbf{1}$,~$\mathbf{2}$,~$\mathbf{3}$, $\mathbf{4}$, $\mathbf{5}$, $\mathbf{6}$,
$\mathbf{8}$, $\mathbf{9}$, $\mathbf{10}$, $\mathbf{12}$,
$\mathbf{15}$, $\mathbf{16}$, $\mathbf{17}$
\end{center}
in the~table in the~end of the~paper.
\end{corollary*}

\begin{proof}
If $\pi_1$ is a fibration into del Pezzo surfaces of degree $6$,
the assertion follows from \cite{Fukuoka2017,Fukuoka2019}, in which case we get the link $\mathbf{15}$.
In the~remaining cases, the~required assertion follows from \cite{Takeuchi2022}.
\end{proof}

Therefore, we may assume that neither $\pi_1$ nor $\pi_2$ is a fibration into del Pezzo surfaces.

\begin{proposition*}
Suppose that $\pi_1$ and $\pi_2$ are birational. Then \eqref{equation:link} is the~link $\mathbf{13}$~in~the~table.
\end{proposition*}

\begin{proof}
Both $Z_1$ and $Z_2$ are (possibly singular) Fano threefolds,
and \mbox{$\mathrm{rk}\,\mathrm{Pic}(Z_1)=\mathrm{rk}\,\mathrm{Pic}(Z_2)=1$}.

Suppose $Z_1$ is smooth so that $\pi_1$ is a contraction of type $E1$ or $E2$ in \cite[Theorem 1.32]{KM97} and $\pi_2$ is a contraction of type $E1$--$E5$. Then possibilities for $h^{1,2}(Z_1)$ are listed in the~two tables presented~above.
Using \cite{CutroneMarshburn}, we obtain all possible values of $h^{1,2}(X_1)$.
Now, using~\eqref{equation:h-1-2}, in combination with the list of Sarkisov links in \cite[Tables 1--7]{CutroneMarshburn} we see, carrying out a case-by-case analysis, that $Z_1\cong Z_2\cong\mathbb{P}^3$,
and $\pi_1$ and $\pi_2$ are the blow ups along smooth rational curves of degree $5$. Alternatively, one can run a short computer program exhausting all the possibilities for $Z_1$ and $Z_2$ and reach the same conclusion.

Therefore, to show that \eqref{equation:link} is the~link $\mathbf{13}$ in the~table it suffices to explain why the rational quintic curves
are not contained in a quadric.
Indeed,
none of these curves are contained in a smooth quadric surface,
because in that case one of the rulings of this quadric will be contracted in the anticanonical model, 
but birational morphisms $\phi_1$ and $\phi_2$ are small by construction.
Furthermore, a degree $5$ smooth rational curve in $\mathbb{P}^3$ is never contained in a singular quadric.

Thus, we may assume that both $Z_1$ and $Z_2$ are singular.
Now, using \cite[Tables 8--9]{CutroneMarshburn}, we get $-K_X^3\in\{2,4\}$.
Hence, if $|-K_X|$ does not have base points, then $X$ is one of the~following threefolds:
\begin{enumerate}
\item sextic hypersurface in $\mathbb{P}(1,1,1,1,3)$,
\item quartic hypersurface in $\mathbb{P}^4$,
\item complete intersection of a quadric cone and a quartic hypersurface in $\mathbb{P}(1,1,1,1,1,2)$.
\end{enumerate}
Indeed, by the Riemann-Roch theorem \cite[Corollary 2.1.14]{IskovskikhProkhorov},
$|-K_X|$ defines a finite map $\phi\colon X \rightarrow \mathbb P^N$ with $N = 3$ 
(respectively $N = 4$) when $-K_X^3 = 2$
(respectively $-K_X^3 = 4$).
We have $\deg(\phi(X))\cdot \deg(\phi)=-K_X^3$. If $-K_X^3=2$, then $\phi$ is a double cover of $\mathbb P^3$ ramified at a sextic hypersurface by Hurwitz's formula, thus giving the first case. If $-K_X^3=4$ we either get that $\deg(\phi)=1$ and $\phi(X)$ is a quartic threefold or $\deg(\phi)=\deg(\phi(X))=2$ and we get the last case. 

By studying the defect, in each of these cases, the~threefold $X$ is factorial
as it follows from \cite{Cheltsov2006,Cheltsov2009,Cheltsov2010,CheltsovPark2010,Shramov},
contradicting our assumption. 

Therefore $|-K_X|$ has base points,
hence using \cite[Theorem 1.1 (i)]{JahnkeRadloff2006}, we see that $-K_X^3=2$,
and $X$ is the complete intersection of a quadric cone and a sextic hypersurface in $\mathbb{P}(1,1,1,1,2,3)$ on variables $x_0, \ldots, x_5$. We can assume that
the quadric
cone is given by $x_0x_1-x_2x_3=0$. 
Then the projection on $x_0, x_2$ coordinates gives (after a small resolution of the singularity) a fibration by del Pezzo surfaces of degree $1$. Similarly, the projection on $x_0,x_3$ coordinates gives another such fibration. 
This  implies that \eqref{equation:link} is the~Sarkisov link $\mathbf{1}$ in the~table,
so that $\pi_2$ is not birational, which contradicts our assumption.
\end{proof}

Thus, we may assume that $\pi_1$ is a conic bundle,
and  either $\pi_2$ is birational, or $\pi_2$ is a conic~bundle.
Then the surface $Z_1$ is smooth \cite[(3.5.1)]{Mori}, which implies that $Z_1=\mathbb{P}^2$, since $X_1$ has Picard rank two.
Let $d_1$ be the~degree of the~discriminant curve of the~conic bundle $\pi_1$.
Then \cite[1.6 Main Theorem]{Sarkisov} implies $0 \leqslant d_1\leqslant 11$;
$d_1=0$ if $\pi_1$ is a $\mathbb{P}^1$-bundle.
By \cite{Beauville}, we get 
\begin{equation}
\label{equation:h12-degree}\tag{$\spadesuit$}
h^{1,2}(X_1)=\frac{d_1(d_1-3)}{2},
\end{equation}
so $d_1\not\in\{1,2\}$. 
Using \eqref{equation:h-1-2} and the~list of possible values of $h^{1,2}$ presented in tables above, we get
$$
d_1\in\{0,3,4,5,7,8\}.
$$
Using the Observation above, for the rest of the proof we will assume that $I=1$. Therefore we have
\begin{equation}
\label{equation:list}
\tag{$\diamondsuit$}
(d,h^{1,2},d_1)\in
\left\{
(6,20,8), (8,14,7), (14,5,5), (18,2,4), (22,0,0),
(22,0,3)
\right\}.
\end{equation}

Let $D_2$ be a Cartier divisor on $X_2$, let $D_1$ be its strict transform on $X_1$,
and let $H_1$ be a sufficiently general surface in  $|\pi_1^*(\mathcal{O}_{\mathbb{P}^2}(1))|$.
Then $D_1\sim_{\mathbb{Q}} a(-K_{X_1})-bH_1$ for some rational numbers $a$ and~$b$.
Moreover, if $d_1\ne 0$, then $a$ and $b$ are integers, because
the conic bundle has no sections
and the Picard group of the generic fiber is generated by its canonical class.
If $d_1=0$, then $2a$ and $2b$ are integers, because the Picard group of the generic fiber is generated by the class of a section. On the other hand we have (e.g. see \cite[Lemma~A.3]{CheltsovRubinstein})
\begin{align*}
-K_{X_1}\cdot D_1^2&=-K_{X_2}\cdot D_2^2,\\
\big(-K_{X_1}\big)^2\cdot D_1&=\big(-K_{X_2}\big)^2\cdot D_2.
\end{align*}
Moreover, we have \cite[Proposition 6]{Cheltsov2004}
\begin{equation}
\label{equation:dagger}
\tag{$\dagger$}
-K_{X_1}^3=d,\;\;
(-K_{X_1})^2\cdot H_1=12-d_1,\;\;
-K_{X_1}\cdot H_1^2=2,\;\; 
H_1^3=0.
\end{equation}
This gives
\begin{equation}
\label{equation:Takeuchi}
\tag{$\heartsuit$}
\left\{\aligned
&da^2-2(12-d_1)ab+2b^2=-K_{X_2}\cdot D_2^2,\\
&da-(12-d_1)b=\big(-K_{X_2}\big)^2\cdot D_2.
\endaligned
\right.
\end{equation}

\begin{lemma*}
Suppose that $\pi_2$ is birational.
Then \eqref{equation:link} is either the~link $\mathbf{11}$ or the~link  $\mathbf{14}$ in the~table.
\end{lemma*}

\begin{proof}
Let $D_2$ be the~$\pi_2$-exceptional surface. Then $a=D_1\cdot H_1^2\geqslant 0$.

If $\pi_2(D_2)$ is a point, then it follows from \cite[Theorem (3.3)]{Mori} that one of the~following cases holds:
\begin{itemize}
\item[$(\mathrm{A})$] $D_2=\mathbb{P}^2$ with normal bundle $\mathcal{O}(-1)$,
\item[$(\mathrm{B})$] $D_2=\mathbb{P}^2$ with normal bundle $\mathcal{O}(-2)$,
\item[$(\mathrm{C})$] $D_2$ is an irreducible quadric surface in $\mathbb{P}^3$
with normal bundle $\mathcal{O}(-1)$.
\end{itemize}
A simple computation using the adjunction formula implies that
$
-K_{X_2}\cdot D_2^2= -2$
and
$$
\big(-K_{X_2}\big)^2\cdot D_2=\left\{\aligned
&4\ \text{in the~case}\ (\mathrm{A}),\\
&1\ \text{in the~case}\ (\mathrm{B}),\\
&2\ \text{in the~case}\ (\mathrm{C}).
\endaligned
\right.
$$
Now, solving \eqref{equation:Takeuchi} for each triple $(d,h^{1,2},d_1)$ listed in \eqref{equation:list},
we see that $2a$ is never a non-negative integer. This shows that $\pi_2(D_2)$ is not a point.

We see that $Z_2$ is a smooth Fano threefold of Picard rank $1$,
and $\pi_2(D_2)$ is a smooth curve in~$Z_2$.
Then it follows from \cite[Theorem~7.14]{JahnkePeternellRadloff2011} and \eqref{equation:h-1-2} that
\eqref{equation:link} is one of the~Sarkisov links $\mathbf{11}$ and  $\mathbf{14}$,
which would complete the~proof of the~lemma.

Note, however, that the~paper \cite{JahnkePeternellRadloff2011} has gaps \cite[Remark~1.18]{CheltsovShramov}.
For instance, the~link in the Construction below
contradicts \cite[Theorem~7.4]{JahnkePeternellRadloff2011},
and few examples constructed in~\cite{Yasutake} contradict~\cite[Proposition 7.2]{JahnkePeternellRadloff2011}.
Keeping this in mind, let us complete the~proof of the~lemma without using \cite[Theorem~7.14]{JahnkePeternellRadloff2011}.

Set $C_2=\pi_2(D_2)$.
Let $d_2=-K_{Z_2}\cdot C_2$, and let $g_2$ be the~genus of the~curve $C_2$. Then as $\pi_2$ is the blow up along a curve on a threefold, we have that
$$
h^{1,2}(Z_2)+g_2=h^{1,2}\in\{0,2,5,14,20\},
$$
where the inclusion follows from \eqref{equation:list}.
As a result, using the classification of smooth Fano threefolds \cite[\S 12.2]{IskovskikhProkhorov}, we get $h^{1,2}(Z_2)\in\{0,2,3,5,7,10,14,20\}$.
In fact, we can say a bit more. Let $e=-K_{Z_2}^3$, let $i$ be the~index of the~Fano threefold~$Z_2$. Then
\begin{itemize}
\item $(e,i)=(64,4)$  $\iff$ $Z_2=\mathbb{P}^3$,
\item $(e,i)=(54,3)$  $\iff$ $Z_2$ is a smooth quadric threefold in $\mathbb{P}^4$.
\end{itemize}
Moreover, the~possible values  $h^{1,2}(Z_2)\leqslant 20$ can be listed as follows.

\begin{center}
\renewcommand\arraystretch{1.4}
\begin{tabular}{|c||c|c|c|c|c|c|c|c|}
  \hline
$(e,i)$             & $(6,1)$ & $(8,1)$ & $(10,1)$ & $(12,1)$ & $(14,1)$, & $(16,1)$ & $(18,1)$ & $(22,1)$\\
  \hline
$h^{1,2}(Z_{2})$    & $20 $   & $14$    & $10$      & $7$     & $5$       & $3$      & $2$      & $0$    \\
  \hline
\end{tabular}
\end{center}

\begin{center}
\renewcommand\arraystretch{1.4}
\begin{tabular}{|c||c|c|c|c|c|c|c|c|c|c|c|c|c|c|}
  \hline
$(e,i)$            & $(16,2)$ & $(24,2)$ & $(32,2)$ & $(40,2)$ & $(54,3)$ & $(64,4)$\\
  \hline
$h^{1,2}(Z_{2})$   &  $10$     & $5$     & $2$      & $0$      & $0$      & $0$\\
  \hline
\end{tabular}
\end{center}
\noindent
This leaves not so many possibilities for the~genus $g_2=h^{1,2}-h^{1,2}(Z_2)$.

One the~other hand, it follows from \cite[Lemma 4.1.2]{IskovskikhProkhorov} that
\begin{align*}
-K_{X_2}\cdot D_2^2&=2g_2-2,\\
(-K_{X_2})^2\cdot D_2&=d_2+2-2g_2,\\
-K_{X_2}^3&=e-2+2g_2-2d_2,
\end{align*}
so that \eqref{equation:Takeuchi} gives
$$
\left\{\aligned
&da^2-2(12-d_1)ab+2b^2=2g_2-2,\\
&da-(12-d_1)b=d_2+2-2g_2,\\
&d=e-2+2g_2-2d_2.
\endaligned
\right.
$$
Now, solving this system of equations for each triple $(d,I,h^{1,2},d_1)$ listed in \eqref{equation:list},
and each possible triple $(e,i,g_2)=(e,i,h^{1,2}-h^{1,2}(Z_2))$, we obtain the~following two cases:
\begin{itemize}
\item[$(\mathrm{I})$] $d=18$, $I=1$, $h^{1,2}=2$, $d_1=4$, $Z_2=\mathbb{P}^3$, $d_2=24$, $g_2=2$, $a=3$, $b=4$;
\item[$(\mathrm{II})$] $d=22$, $I=1$, $h^{1,2}=0$, $d_1=3$, $Z_2$ is a smooth quadric in $\mathbb{P}^4$, $d_2=15$, $g_2=0$, $a=3$, $b=4$;
\end{itemize}
In the~case $(\mathrm{I})$, \eqref{equation:link} is the~link $\mathbf{11}$ in the~table.
In the~case $(\mathrm{II})$, \eqref{equation:link} is the~link $\mathbf{14}$ in the~table.
\end{proof}

Therefore, we may assume that $\pi_2$ is also a conic bundle and $Z_2=\mathbb{P}^2$.
Let $d_2$ be the~discriminant curve of the~conic bundle $\pi_2$.
Using \eqref{equation:h12-degree}
and $h^{1,2}(X_1) = h^{1,2}(X_2)$ we obtain that either 
$d_1 = d_2$
or $d_1, d_2 \in \{0, 3\}$.
Now, we let $D_2$ be a general surface in $|\pi_2^*(\mathcal{O}_{\mathbb{P}^2}(1))|$.
Then \eqref{equation:Takeuchi} simplifies as
$$
\left\{\aligned
&da^2-2(12-d_1)ab+2b^2=2,\\
&da-(12-d_1)b=12-d_2.
\endaligned
\right.
$$
Solving these equations for each quadruple $(d,h^{1,2},d_1)$ listed in \eqref{equation:list},
we get the~following cases:
\begin{enumerate}
\item[($\mathrm{1}$)] $a=0$, $b=-1$;
\item[($\mathrm{2}$)] $d=14$, $I=1$, $h^{1,2}=5$, $d_1=d_2=5$, $a=1$, $b=1$.
\end{enumerate}
In the~case ($\mathrm{1}$), the~composition $\varphi_2\circ\varphi_1^{-1}$ is biregular. This contradicts our initial assumption.
So, the~case ($\mathrm{2}$) holds. Then \eqref{equation:link} is the~link $\mathbf{7}$~in~the~table, which proves the~theorem.

\medskip

Let us conclude this paper by showing that the~Sarkisov link $\mathbf{7}$~in~the~table is always obtained using the following:

\begin{construction*}[{\cite[\S~3.4 Case $4^o$]{Prokhorov2023}}]
Let $\overline{E}=\{z_1=z_2=0\}\subset\mathbb{P}^2_{x_1,y_1,z_1}\times\mathbb{P}^2_{x_2,y_2,z_2}$,
and let
$$
\overline{X}=\big\{z_1f(x_1,y_1,z_1;x_2,y_2,z_2)=z_2g(x_1,y_1,z_1;x_2,y_2,z_2)\big\},
$$
where $f$ and $g$ are sufficiently general polynomials of bi-degrees $(1,2)$ and $(2,1)$, respectively.
Then $\overline{X}$ is a singular Verra threefold (a bidegree $(2,2)$ threefold in $\mathbb P^2\times \mathbb P^2$) with $5$ nodes.
Note that $\overline{E}\cong\mathbb{P}^1\times\mathbb{P}^1$, $\overline{E}\subset \overline{X}$ and
$$
\mathrm{Sing}(\overline{X})=\big\{z_1=z_2=f=g=0\big\}\subset\overline{E}.
$$
Let $\rho\colon\mathbb{P}^2_{x_1,y_1,z_1}\times\mathbb{P}^2_{x_2,y_2,z_2}\dashrightarrow\mathbb{P}^4_{x,y,z,t,w}$ be the~rational map given by
$$
\big([x_1:y_1:z_1],[x_2:y_2:z_2]\big)\mapsto\big[x_1z_2:y_1z_2:x_2z_1:y_2z_1:z_1z_2\big].
$$
Then $\rho$ is birational, and the~inverse map $\rho^{-1}$ is given by $[x:y:z:t:w]\mapsto([x:y:w],[z:t:w])$.
Let $\xi\colon W\to\mathbb{P}^2_{x_1,y_1,z_1}\times\mathbb{P}^2_{x_2,y_2,z_2}$ be the~blow up along the~surface~$\overline{E}$
and
let $\mathscr{E}$ be its exceptional divisor.
Let $\overline{G}_1=\{z_1=0\}$ and $\overline{G}_2=\{z_2=0\}$, and let $G_1$ and $G_2$ be the proper transforms on $W$ of $\overline{G}_1$ and $\overline{G}_2$.
Then we have the~following commutative diagram:
$$
\xymatrix{
&W\ar[dl]_{\xi}\ar[drrr]^{\theta}&&&\\
\mathbb{P}^2_{x_1,y_1,z_1}\times\mathbb{P}^2_{x_2,y_2,z_2}\ar@{-->}[rrrr]_{\rho}&&&&\mathbb{P}^4_{x,y,z,t,w}&}
$$
where $\theta$ blows down $G_1$ and $G_2$ to the~lines $\ell_1=\{z=t=w=0\}$ and $\ell_2=\{x=y=w=0\}$.
Note~that $\theta(\mathscr{E})$ is the~hyperplane $\{w=0\}$ --- the~unique hyperplane containing the~lines $\ell_1$ and~$\ell_2$.
Set~$V=\rho(\overline{X})$. Then $V$ is a smooth cubic threefold in $\mathbb{P}^4_{x,y,z,t,w}$. Moreover, we have
$$
V=\big\{f(x,y,w;z,t,w)=g(x,y,w;z,t,w)\big\}\subset\mathbb{P}^4_{x,y,z,t,w}.
$$
Now, let $\widehat{X}$ be the~strict transform of the~threefold $\overline{X}$ on $W$,
let $\varsigma\colon \widehat{X}\to\overline{X}$ be the~morphism induced by~$\xi$,
and let $\nu\colon \widehat{X}\to V$ be the~morphism induced by~$\theta$.
Then $\widehat{X}$ is smooth, $\varsigma$ is a small projective resolution,
and we have the~following commutative diagram:
$$
\xymatrix{
&\widehat{X}\ar[dl]_{\varsigma}\ar[dr]^{\nu}&\\
\overline{X}\ar@{-->}[rr]_{\rho\vert_{\overline{X}}}&&V}.
$$
Note that $\nu$ is the~blow up of the~cubic threefold $V$ along the~lines $\ell_1$ and $\ell_2$.
Let $\widehat{E}=\mathscr{E}\vert_{\widehat{X}}$. Then
\begin{itemize}
\item the~induced map $\varsigma\vert_{\widehat{E}}\colon \widehat{E}\to\overline{E}$ is the blow up at the~points in $\mathrm{Sing}(\overline{X})$,
\item $\widehat{E}$ is isomorphic to a smooth cubic surface,
\item $\nu(\widehat{E})$ is the~hyperplane section $\{w=0\}\cap V$.
\end{itemize}
Now, we extend the~last commutative diagram to the~following commutative diagram:
$$
\xymatrix{
&&V&&\\
V_1\ar[d]_{\upsilon_1}\ar[urr]^{\psi_1}&&\widehat{X}\ar[d]_{\varsigma}\ar[u]_{\nu}\ar[ll]^{\nu_2}\ar[rr]_{\nu_1}&&V_2\ar[d]^{\upsilon_2}\ar[ull]_{\psi_2}\\
\mathbb{P}^2_{x_1,y_1,z_1}&&\overline{X}\ar[rr]_{\mathrm{pr}_2}\ar[ll]^{\mathrm{pr}_1}&&\mathbb{P}^2_{x_2,y_2,z_2}}
$$
Here $\psi_1$ and $\psi_2$ are the blow ups along the~lines $\ell_1$ and $\ell_2$, respectively,
$\nu_1$ and $\nu_2$ are the blow ups along the~strict transforms of the~lines $\ell_1$ and $\ell_2$, respectively,
both $\upsilon_1$ and $\upsilon_2$  are standard conic bundles~\cite{Prokhorov2018},
and~$\mathrm{pr}_1$ and $\mathrm{pr}_2$ are the natural projections.
Let $\Delta_1$ and $\Delta_2$ be the~discriminant curves of the~conic bundles $\upsilon_1$ and $\upsilon_2$, respectively.
Then $\Delta_1$ and $\Delta_2$ are quintic curves with at most nodal singularities.
Since $\varsigma$ is a flopping contraction, there exists a composition of flops $\chi\colon\widehat{X}\dasharrow\widetilde{X}$ of the 5~curves contracted by $\varsigma$ (this is the only projective flop which exists because the relative Picard number of $\varsigma$ equals $1$).
Then $\widetilde{X}$ is smooth and projective, and we have another commutative diagram:
$$
\xymatrix{
\widetilde{X}\ar[dr]_{\sigma}&&\widehat{X}\ar@{-->}[ll]_{\chi}\ar[dl]_{\varsigma}\ar[dr]^{\nu}&\\
&\overline{X}\ar@{-->}[rr]_{\rho\vert_{\overline{X}}}&&V}
$$
where $\sigma$ is a small resolution. Let $E=\chi(\widehat{E})$.
Then $\chi$ induces a~morphism $\widehat{E}\to E$ that blows down all five curves contracted by $\varsigma$,
which implies that $\sigma$ induces an~isomorphism $E\cong\overline{E}\cong\mathbb{P}^1\times\mathbb{P}^1$.
Note that $E\vert_{E}\sim\mathcal{O}_{E}(-1,-1)$,
and there exists a birational morphism $\eta\colon\widetilde{X}\to X$ that blows down the~surface $E$ to an~ordinary double point of the~threefold $X$.
We have $-K_X^3=-K_{\overline{X}}^3-2=14$ and
$$
1=\mathrm{rk}\,\mathrm{Pic}(X)<\mathrm{rk}\,\mathrm{Cl}(X)=1+\big|\mathrm{Sing}(X)\big|=2.
$$
Therefore, the~threefold $X$ is a non-factorial nodal Fano threefold that has one node.
We complete the picture with the following commutative diagram
$$
\xymatrix{
&&X&&\\
X_1\ar[d]_{\pi_1}\ar[rru]^{\phi_1}&&\widetilde{X}\ar[rr]^{\varphi_2}\ar[ll]_{\varphi_1}\ar[u]_{\eta}\ar[d]_{\sigma}&&X_2\ar[llu]_{\phi_2}\ar[d]^{\pi_2}&\\
\mathbb{P}^2_{x_1,y_1,z_1}&&\overline{X}\ar[rr]_{\mathrm{pr}_2}\ar[ll]^{\mathrm{pr}_1}&&\mathbb{P}^2_{x_2,y_2,z_2}\\
V_1\ar[u]^{\upsilon_1}\ar[drr]_{\psi_1}&&\widehat{X}\ar[u]_{\varsigma}\ar[d]_{\nu}\ar[ll]^{\nu_2}\ar[rr]_{\nu_1}&&V_2\ar[u]_{\upsilon_2}\ar[dll]^{\psi_2}\\
&&V&&}
$$
where $\phi_1$ and $\phi_2$ are two small resolutions such that the~composition $\phi_1^{-1}\circ\phi_2$ is an Atiyah flop,
both~$\varphi_1$ and $\varphi_2$ are contractions of the~surface $E$ to curves,
$\pi_1$ and $\pi_2$ are standard conic bundles whose discriminant curves are $\Delta_1$ and $\Delta_2$, respectively.
Note that $X$ is irrational as it is birational to a smooth 
cubic threefold \cite{ClemensGriffiths}, and
$$
h^{1,2}(X_1)=h^{1,2}(X_2)=h^{1,2}(\widetilde{X})=h^{1,2}(\widehat{X})=h^{1,2}(V)=5.
$$
Instead of using the~Verra threefold $\overline{X}$ containing $\overline{E}$,
we can construct the~nodal threefold $X$ using the~birational map $\rho^{-1}$,
and the~smooth cubic threefold $V$ containing the~lines $\ell_1$ and $\ell_2$.
\end{construction*}

Now consider~link $\mathbf{7}$ in the~table:
$Z_1 = Z_2 = \mathbb{P}^2$, and both 
$\pi_1$
and $\pi_2$ are conic bundles
with discriminant curves of degree $5$.
Let $C_1$ and $C_2$ be the~curves contracted by $\phi_1$ and $\phi_2$, respectively.

Recall that we denote by $H_1$ (respectively $D_2$)
the pullback of the ample generator by $\pi_1$
from $Z_1$ (respectively by $\pi_2$ from $Z_2$), and
$D_1$ is the divisor corresponding to $D_2$ on $X_1$ under flop.
Then it follows from the calculations above (see case (2) before the Construction)
that $D_1\sim -K_{X_1}-H_1$.
We have
$$
-1=\big(-K_{X_1}-H_1\big)^3=D_1^3=D_2^3-\big(D_2\cdot C_2\big)^3=-\big(D_2\cdot C_2\big)^3,
$$
where we used \eqref{equation:dagger} in the first equality,
and
\cite[Lemma~A.3]{CheltsovRubinstein} in the third one.
It follows that $D_2\cdot C_2=1$.
Similarly, we get $H_1\cdot C_1=1$.

Let $h_1 = \varphi_1^*(H_1)$ and $h_2 = \varphi_2^*(D_2)$.
A simple computation using $D_2 \cdot C_2 = 1$ implies that
\[
\varphi_1^*(D_1) \sim h_2 + E.
\]
Thus we can express the canonical class $-K_{\widetilde{X}}$ in terms of $h_1$ and $h_2$ as follows
$$
-K_{\widetilde{X}}\sim -\varphi_1^*(K_{X_1}) - E 
\sim \varphi_1^*(H_1 + D_1) - E \sim
h_1 + h_2.
$$
Note that $-K_{\widetilde{X}}^3=12$, $h^{1,2}(\widetilde{X})=5$ and $\mathrm{rk}\,\mathrm{Pic}(\widetilde{X})=3$,
which implies that $-K_{\widetilde{X}}$ is not ample,
because smooth Fano threefolds with these  invariants do not exist \cite[Table 3]{MoriMukai}.

Combining $\pi_1\circ\varphi_1$ and $\pi_2\circ\varphi_2$, we obtain a morphism $\widetilde{X}\to\mathbb{P}^2\times\mathbb{P}^2$.
Let $\overline{X}$ be its image, and let $\sigma\colon\widetilde{X}\to\overline{X}$ be the~induced~morphism.

\begin{claim*}
The threefold $\overline{X} \subset \P^2 \times \P^2$ is a 
divisor of bidegree $(2,2)$ with terminal singularities, containing a linearly embedded surface
$\P^1 \times \P^1$,
and $\sigma$ is a small resolution.
 \end{claim*}

Therefore $X$ is obtained by taking a small resolution
of the singular
Verra threefold $\overline{X}$
containing a divisor $\P^1 \times \P^1$ 
as in  Construction above.

\begin{proof}
The threefold 
$\overline{X}$ is a divisor
of bidegree $(e_1, e_2)$
in $\P^2 \times \P^2$,
with $e_1, e_2 > 0$ because
$\overline{X}$ dominates both factors.
We have
\[
12 = (h_1+h_2)^3 = \deg(\sigma) \deg(\overline{X}) = 3\deg(\sigma)(e_1+e_2)
\]
This implies that either 
$\deg(\sigma) = 1$ and
$e_1 + e_2 = 4$, in which case $e_1 = e_2 = 2$
because the two projections give rise to conic bundle structures on $\widetilde{X}$, or 
$\deg(\sigma) = 2$
and $e_1 + e_2 = 2$
so that $e_1 = e_2 = 1$ because $e_1, e_2 > 0$. 
In other words
\begin{itemize}
\item either $\overline{X}$ is a divisor of degree $(2,2)$, and $\sigma$ is birational,
\item or $\overline{X}$ is a divisor of degree $(1,1)$, and $\sigma$ is generically two-to-one.
\end{itemize}
In the~former case, 
$\sigma$ is crepant,
and
it follows from the~subadjunction formula that the~threefold $\overline{X}$ is normal.
In the~latter case, the~threefold $\overline{X}$ is also normal, because
there are only two
isomorphism classes
of irreducible $(1,1)$ divisors in $\P^2 \times \P^2$:
one is smooth and the other has one node. 

Set  $\overline{E}=\sigma(E)$.  Let $\mathrm{pr}_1\colon\overline{X}\to\mathbb{P}^2$ and $\mathrm{pr}_2\colon\overline{X}\to\mathbb{P}^2$
be the~projections to the~first and the~second factors of the~fourfold $\mathbb{P}^2\times\mathbb{P}^2$, respectively.
Then $\mathrm{pr}_1(\overline{E})$ and $\mathrm{pr}_2(\overline{E})$ are lines by $H_1 \cdot C_1 = D_2 \cdot C_2 = 1$,
so we can choose coordinates $([x_1:y_1:z_1],[x_2:y_2:z_2])$ on $\mathbb{P}^2\times\mathbb{P}^2$ such that
$$
\overline{E}=\big\{z_1=z_2=0\big\}.
$$
Since $\overline{E}\subset\overline{X}$, we see that $\overline{X}$ is singular.
Note also that $\sigma$ induces an isomorphism $E\cong\overline{E} = \mathbb{P}^1 \times \mathbb{P}^1$.

Divisor classes $h_1$, $h_2$ and $E$ generate the~group $\mathrm{Pic}(\widetilde{X})$.
We have
\[
h_1^2 \cdot h_2 = h_1 \cdot h_2^2 = 2, \quad h_1 \cdot h_2 \cdot E = 1, \quad
h_1^2 \cdot E = h_2^2 \cdot E = 0.
\]

Assume that $\sigma$ contracts a divisor 
$F\sim a_1 h_1+a_2h_2+a_3E$.
Then we have
\begin{align*}
2a_2&=F\cdot h_1^2 =0,\\
2a_1&=F\cdot h_2^2 =0,\\
2a_1+2a_2+a_3&=F\cdot h_1 \cdot h_2=0,
\end{align*}
which gives $a_1=0$, $a_2=0$, $a_3=0$. This shows that  $\sigma$ does not contract any divisors.

The Stein factorization of $\sigma$ is the~following commutative diagram:
$$
\xymatrix{
\widetilde{X}\ar[rd]_{\sigma}\ar[rr]^{\alpha}&&\widehat{X}\ar[dl]^{\beta}\\
&\overline{X}&}
$$
where $\alpha$ is a birational morphism, and $\beta$ is either an~isomorphism or a~(ramified) double cover.
Since $\sigma$ does not contract divisors and $-K_{\widetilde{X}}$ is not ample,
we see that $\alpha$ is a flopping contraction, and $\widehat{X}$ has terminal Gorenstein singularities.
We must show that $\beta$ is an~isomorphism.

Suppose $\beta$ is a double cover. Its Galois involution induces a~birational involution $\tau\in\mathrm{Bir}(\widetilde{X})$.
Then $\tau$ induces an action $\tau_*$ on $\mathrm{Pic}(\widetilde{X}) =
\mathrm{Cl}(\widehat{X})$ such that $\tau_*h_1 \sim h_1$, $\tau_*h_2 \sim h_2$,
and
$$
\tau_*(E)\sim b_1 h_1+b_2 h_2+b_3E
$$
for some integers $b_1$, $b_2$, $b_3$. Then
\begin{align*}
2b_2&=\tau_*(E)\cdot h_1^2=E\cdot h_1^2=0,\\
2b_1&=\tau_*(E)\cdot  h_2^2=E\cdot  h_2^2=0,\\
2b_1+2b_2+b_3&=\tau_*(E)\cdot h_1 \cdot h_2=E\cdot h_1 \cdot h_2=1,
\end{align*}
which gives $b_1=0$, $b_2=0$, $b_3=1$, so $\tau_*(E)\sim E$, which gives $\tau(E)=E$, since $E$ is $\eta$-exceptional.

Since $\tau(E)=E$ and $\sigma$ induces an isomorphism $E\cong\overline{E}$,
we see that the~surface $\overline{E}$ is contained in the~branch divisor of the~double cover $\beta$.
On the other hand, $\overline{E}$ can not be equal to this branch divisor
by degree reasons, thus the branch divisor is reducible.
This implies that $\widehat{X}$ has non-isolated singularities,
which is impossible, since  $\widehat{X}$ has terminal singularities.
Thus, we see that $\beta$ is an isomorphism.

We see that $\overline{X}$ is a 
singular
divisor in $\mathbb{P}^2\times\mathbb{P}^2$ of degree $(2,2)$, containing $\overline{E}$
and $\sigma$ is a small resolution.
\end{proof}

\bigskip

\textbf{Acknowledgements.}
This paper was written during our research visit to the G\"okova Geometry Topology Institute in April 2023.
We are very grateful to the~institute for its warm hospitality and to Tiago Duarte Guerreiro and Kento Fujita, who joined us in the visit on a separate project.

Ivan Cheltsov was supported by EPSRC grant EP/V054597/1.

Jesus Martinez-Garcia was supported by EPSRC grant EP/V055399/1.

Evgeny Shinder was supported by EPSRC grant EP/T019379/1 and ERC Synergy grant 854361.

Igor Krylov was supported by IBS-R003-D1 grant.

\newpage

\begin{landscape}

\renewcommand\arraystretch{1.8}
\begin{longtable}{|c||c|c|c|c|c|c|}
\caption*{Table describing all possibilities for the~Sarkisov link \eqref{equation:link}.}
\endfirsthead
\endhead
\hline
\textnumero& $d$& $I$ & $h^{1,2}$ & $\pi_1\colon X_1\to Z_1$ &  $\pi_2\colon X_2\to Z_2$ & References\\
\hline
\hline
\shortstack{\\ $\mathbf{1}$\\$\quad$\\$\quad$\\$\quad$} &
\shortstack{\\$2$\\$\quad$\\$\quad$\\$\quad$} &  \shortstack{\\$1$\\$\quad$\\$\quad$\\$\quad$} &  \shortstack{\\$52$\\$\quad$\\$\quad$\\$\quad$} &
\shortstack{\\$Z_1=\mathbb{P}^1$,\\ $\pi_1$ is a fibration into \\ del Pezzo surfaces of degree $1$.}&
\shortstack{\\$Z_2=\mathbb{P}^1$,\\ $\pi_2$ is is a fibration into \\ del Pezzo surfaces of degree $1$.} &
\shortstack{\\\cite{Grinenko,Grinenko2006,JahnkeRadloff2006},\\ \cite[(2.5.2)]{Takeuchi2022}.\\$\quad$} \\
\hline
\shortstack{\\ $\mathbf{2}$\\$\quad$\\$\quad$\\$\quad$\\$\quad$\\$\quad$} &
\shortstack{\\$6$\\$\quad$\\$\quad$\\$\quad$\\$\quad$\\$\quad$} &  \shortstack{\\$1$\\$\quad$\\$\quad$\\$\quad$\\$\quad$\\$\quad$} &  \shortstack{\\$20$\\$\quad$\\$\quad$\\$\quad$\\$\quad$\\$\quad$} &
\shortstack{\\$Z_1=\mathbb{P}^1$,\\ $\pi_1$ is a fibration into \\ del Pezzo surfaces of degree $2$.\\$\quad$\\$\quad$\\$\quad$}&
\shortstack{\\$Z_2$ is a del Pezzo threefold of degree $1$ \\ that has one singular double point,\\ $\pi_2$ is the blow up at the~singular point.\\$\quad$\\$\quad$\\$\quad$} &
\shortstack{\\$\quad$\\\cite[Proposition 5.6]{CheltsovShramovPrzyjalkowski2005},\\\cite{Grinenko,Grinenko2006},\\\cite[Example~4.3]{Prokhorov2019},\\\cite[(2.7.3)]{Takeuchi2022}.}\\
\hline
\shortstack{\\ $\mathbf{3}$\\$\quad$\\$\quad$\\$\quad$\\$\quad$\\$\quad$\\$\quad$} &
\shortstack{\\$8$\\$\quad$\\$\quad$\\$\quad$\\$\quad$\\$\quad$\\$\quad$} &  \shortstack{\\$1$\\$\quad$\\$\quad$\\$\quad$\\$\quad$\\$\quad$\\$\quad$} &  \shortstack{\\$14$\\$\quad$\\$\quad$\\$\quad$\\$\quad$\\$\quad$\\$\quad$} &
\shortstack{\\$Z_1=\mathbb{P}^1$,\\ $\pi_1$ is a fibration into cubic surfaces.\\$\quad$\\$\quad$\\$\quad$}&
\shortstack{\\$Z_2\cong \mathbb{P}^2$,\\ $\pi_2$ is a conic bundle \\ with septic discriminant curve.\\$\quad$} &
\shortstack{\\$\quad$\\$\quad$\\\cite[Proposition 5.9]{CheltsovShramovPrzyjalkowski2005},\\\cite[Example 4.6]{Prokhorov2019},\\\cite[(2.9.4)]{Takeuchi2022}.\\$\quad$} \\
\hline
\shortstack{\\ $\mathbf{4}$\\$\quad$\\$\quad$\\$\quad$} &
\shortstack{\\$10$\\$\quad$\\$\quad$\\$\quad$} &  \shortstack{\\$1$\\$\quad$\\$\quad$\\$\quad$} &  \shortstack{\\$10$\\$\quad$\\$\quad$\\$\quad$} &
\shortstack{\\$Z_1=\mathbb{P}^1$,\\ $\pi_1$ is a fibration into cubic surfaces.\\$\quad$\\$\quad$}&
\shortstack{\\$\quad$\\$\quad$\\$Z_2$ is a smooth del Pezzo threefold of degree $2$,\\ $\pi_2$ is the blow up along a smooth rational curve\\ that has anticanonical degree $2$.} &
\shortstack{\\\cite[Example~1.11]{CheltsovShramovPrzyjalkowski2005},\\ \cite[\S~3.12 Case $11^o$]{Prokhorov2023},\\\cite[(2.9.3)]{Takeuchi2022}.}\\
\hline
\shortstack{\\ $\mathbf{5}$\\$\quad$\\$\quad$\\$\quad$} &
\shortstack{\\$12$\\$\quad$\\$\quad$\\$\quad$} &  \shortstack{\\$1$\\$\quad$\\$\quad$\\$\quad$} &  \shortstack{\\$7$\\$\quad$\\$\quad$\\$\quad$} &
\shortstack{\\$Z_1=\mathbb{P}^1$,\\ $\pi_1$ is a fibration into \\ quartic del Pezzo surfaces.}&
\shortstack{\\$Z_2\cong \mathbb{P}^3$,\\ $\pi_2$ is the blow up along a smooth \\ curve of degree $8$ and genus $7$.} &
\shortstack{\\\cite[Proposition~6.5]{JahnkePeternellRadloff2011},\\ \cite[(2.11.5)]{Takeuchi2022}.\\$\quad$} \\
\hline
\shortstack{\\ $\mathbf{6}$\\$\quad$\\$\quad$\\$\quad$} &
\shortstack{\\$14$\\$\quad$\\$\quad$\\$\quad$} &  \shortstack{\\$1$\\$\quad$\\$\quad$\\$\quad$} &  \shortstack{\\$5$\\$\quad$\\$\quad$\\$\quad$} &
\shortstack{\\$Z_1=\mathbb{P}^1$,\\ $\pi_1$ is a fibration into \\ quartic del Pezzo surfaces.}&
\shortstack{\\$Z_2$ is a smooth cubic threefold,\\ $\pi_2$ is the blow up at a smooth conic.\\$\quad$} &
\shortstack{\\\cite[Proposition~6.5]{JahnkePeternellRadloff2011},\\ \cite[\S~3.13 Case $12^o$]{Prokhorov2023},\\\cite[(2.11.4)]{Takeuchi2022}.} \\
\hline
\shortstack{\\ $\mathbf{7}$\\$\quad$\\$\quad$\\$\quad$\\$\quad$} &
\shortstack{\\$14$\\$\quad$\\$\quad$\\$\quad$\\$\quad$} &  \shortstack{\\$1$\\$\quad$\\$\quad$\\$\quad$\\$\quad$} &  \shortstack{\\$5$\\$\quad$\\$\quad$\\$\quad$\\$\quad$} &
\shortstack{\\$\quad$\\$Z_1=\mathbb{P}^2$,\\ $\pi_1$ is a conic bundle \\ with quintic discriminant curve.\\$\quad$}&
\shortstack{\\$\quad$\\$Z_2=\mathbb{P}^2$,\\ $\pi_1$ is a conic bundle \\ with quintic discriminant curve.\\$\quad$} &
\shortstack{\\$\quad$\\$\quad$\\\cite[\S~3.4 Case $4^o$]{Prokhorov2023}, \\ Construction and\\ Claim in this paper.} \\
\hline
\shortstack{\\ $\mathbf{8}$\\$\quad$\\$\quad$\\$\quad$} &
\shortstack{\\$16$\\$\quad$\\$\quad$\\$\quad$} &  \shortstack{\\$1$\\$\quad$\\$\quad$\\$\quad$} &  \shortstack{\\$3$\\$\quad$\\$\quad$\\$\quad$} &
\shortstack{\\$\quad$\\$Z_1=\mathbb{P}^1$,\\ $\pi_1$ is a fibration into \\ quintic del Pezzo surfaces.}&
\shortstack{\\$\quad$\\$Z_2$ is a smooth quadric in $\mathbb{P}^4$,\\ $\pi_2$ is the blow up along a smooth \\ curve of degree $7$ and genus $3$.} &
\shortstack{\\\cite[Proposition~6.5]{JahnkePeternellRadloff2011},\\ \cite[(2.13.4)]{Takeuchi2022}.\\$\quad$}\\
\hline
\shortstack{\\ $\mathbf{9}$\\$\quad$\\$\quad$\\$\quad$\\$\quad$} &
\shortstack{\\$16$\\$\quad$\\$\quad$\\$\quad$\\$\quad$} &  \shortstack{\\$1$\\$\quad$\\$\quad$\\$\quad$\\$\quad$} &  \shortstack{\\$3$\\$\quad$\\$\quad$\\$\quad$\\$\quad$} &
\shortstack{\\$Z_1=\mathbb{P}^1$,\\ $\pi_1$ is a quadric bundle\\$\quad$\\$\quad$\\$\quad$}&
\shortstack{\\$\quad$\\$\quad$\\$Z_2=\mathbb{P}^1$,\\ $\pi_2$ is a fibration into \\ quartic del Pezzo surfaces.} &
\shortstack{\\$\quad$\\$\quad$\\ \cite[Example~4.9]{Book},\\ \cite[(2.3.8)]{Takeuchi2022},\\\cite[(2.11.2)]{Takeuchi2022}. }\\
\hline
\shortstack{\\ $\mathbf{10}$\\$\quad$\\$\quad$\\$\quad$} &
\shortstack{\\$18$\\$\quad$\\$\quad$\\$\quad$} &  \shortstack{\\$1$\\$\quad$\\$\quad$\\$\quad$} &  \shortstack{\\$2$\\$\quad$\\$\quad$\\$\quad$} &
\shortstack{\\$\quad$\\$\quad$\\$Z_1=\mathbb{P}^1$,\\ $\pi_1$ is a fibration into \\ quintic del Pezzo surfaces.}&
\shortstack{\\$\quad$\\$\quad$\\$Z_2$ is a smooth complete \\ intersection of two quadrics in $\mathbb{P}^5$,\\ $\pi_2$ is the blow up along a twisted cubic.} &
\shortstack{\\\cite[Proposition~6.5]{JahnkePeternellRadloff2011},\\ \cite[(2.13.3)]{Takeuchi2022}.\\$\quad$\\$\quad$}\\
\hline
\shortstack{\\ $\mathbf{11}$\\$\quad$\\$\quad$\\$\quad$\\$\quad$}  &
\shortstack{\\$18$\\$\quad$\\$\quad$\\$\quad$\\$\quad$} &  \shortstack{\\$1$\\$\quad$\\$\quad$\\$\quad$\\$\quad$} &  \shortstack{\\$2$\\$\quad$\\$\quad$\\$\quad$\\$\quad$} &
\shortstack{\\$Z_1\cong \mathbb{P}^2$,\\ $\pi_1$ is a conic bundle\\ with quartic discriminant curve.} &
\shortstack{\\$Z_2=\mathbb{P}^3$,\\ $\pi_2$ is the blow up along a smooth \\ curve of degree $6$ and genus $2$.}&
\shortstack{\\$\quad$\\\cite[Example~4.8]{BlancLamy},\\ \cite[Theorem~7.14]{JahnkePeternellRadloff2011},\\
Lemma in this paper.}\\
\hline
\shortstack{\\ $\mathbf{12}$\\$\quad$\\$\quad$\\$\quad$}  &
\shortstack{\\$22$\\$\quad$\\$\quad$\\$\quad$} &  \shortstack{\\$1$\\$\quad$\\$\quad$\\$\quad$} &  \shortstack{\\$0$\\$\quad$\\$\quad$\\$\quad$} &
\shortstack{\\$\quad$\\$\quad$\\$Z_1=\mathbb{P}^1$,\\ $\pi_1$ is a fibration into \\ quintic del Pezzo surfaces.}&
\shortstack{\\$Z_2\cong \mathbb{P}^2$,\\ $\pi_2$ is a $\mathbb{P}^1$-bundle.\\$\quad$\\$\quad$} &
\shortstack{\\$\quad$\\$\quad$\\\cite[(IV)]{Prokhorov16},\\\cite[(2.13.1)]{Takeuchi2022}.\\ $\quad$} \\
\hline
\shortstack{\\ $\mathbf{13}$\\$\quad$\\$\quad$\\$\quad$\\$\quad$} &
\shortstack{\\$22$\\$\quad$\\$\quad$\\$\quad$\\$\quad$} &  \shortstack{\\$1$\\$\quad$\\$\quad$\\$\quad$\\$\quad$} &  \shortstack{\\$0$\\$\quad$\\$\quad$\\$\quad$\\$\quad$} &
\shortstack{\\$\quad$\\$Z_1=\mathbb{P}^3$,\\ $\pi_1$ is the blow up along a smooth \\ rational curve of degree $5$\\ that is not contained in a quadric.}&
\shortstack{\\$\quad$\\$Z_2=\mathbb{P}^3$,\\ $\pi_1$ is the blow up along a smooth \\ rational curve of degree $5$\\ that is not contained in a quadric.} &
\shortstack{\cite[Proposition~2.11]{CutroneMarshburn},\\\cite[(I)]{Prokhorov16}.\\  $\quad$\\ $\quad$\\$\quad$}\\
\hline
\shortstack{\\ $\mathbf{14}$\\$\quad$\\$\quad$\\$\quad$}  &
\shortstack{\\$22$\\$\quad$\\$\quad$\\$\quad$} &  \shortstack{\\$1$\\$\quad$\\$\quad$\\$\quad$} &  \shortstack{\\$0$\\$\quad$\\$\quad$\\$\quad$} &
\shortstack{\\$\quad$\\$Z_1\cong \mathbb{P}^2$,\\ $\pi_1$ is a conic bundle\\ with cubic discriminant curve.} &
\shortstack{\\$\quad$\\$Z_2$ is a smooth quadric threefold,\\ $\pi_2$ is the blow up along a smooth \\ rational quintic curve.}&
\shortstack{\\$\quad$\cite[Theorem~7.14]{JahnkePeternellRadloff2011},\\\cite[(II)]{Prokhorov16},\\
Lemma in this paper.\\ $\quad$}\\
\hline
\shortstack{\\ $\mathbf{15}$\\$\quad$\\$\quad$\\$\quad$}  &
\shortstack{\\$22$\\$\quad$\\$\quad$\\$\quad$} &  \shortstack{\\$1$\\$\quad$\\$\quad$\\$\quad$} &  \shortstack{\\$0$\\$\quad$\\$\quad$\\$\quad$} &
\shortstack{\\$\quad$\\$Z_1\cong \mathbb{P}^1$,\\ $\pi_1$ is a fibration into\\ sextic del Pezzo surfaces.} &
\shortstack{\\$\quad$\\$Z_2\cong V_5$,\\ $\pi_2$ is the blow up along\\ a rational quartic curve.}&
\shortstack{\\\cite[Proposition~6.5]{JahnkePeternellRadloff2011},\\ \cite[(III)]{Prokhorov16}.\\$\quad$}\\
\hline
\shortstack{\\ $\mathbf{16}$\\$\quad$}  &
\shortstack{\\$40$\\$\quad$} &  \shortstack{\\$2$\\$\quad$} &  \shortstack{\\$0$\\$\quad$} &
\shortstack{\\$Z_1=\mathbb{P}^1$,\\ $\pi_1$ is a quadric bundle.}&
\shortstack{\\$Z_2=\mathbb{P}^2$,\\ $\pi_2$ is a $\mathbb{P}^1$-bundle.} &
\shortstack{\\$\quad$\\\cite[Theorem~3.5]{JahnkePeternell},\\\cite[(2.3.2)]{Takeuchi2022}.}\\
\hline
\shortstack{\\ $\mathbf{17}$\\$\quad$}  &
\shortstack{\\$54$\\$\quad$} &  \shortstack{\\$3$\\$\quad$} &  \shortstack{\\$0$\\$\quad$} &
\shortstack{\\$Z_1=\mathbb{P}^1$,\\ $\pi_1$ is a $\mathbb{P}^2$-bundle.}&
\shortstack{\\$Z_2=\mathbb{P}^1$,\\ $\pi_2$ is a $\mathbb{P}^2$-bundle.}&
\shortstack{\\ Example in this paper.\\$\quad$}\\
\hline
\end{longtable}
\end{landscape}

\end{document}